\def\timestamp{%
Time-stamp: <C-embedding.tex: Saturday 24-08-2024 at 16:42:22 (cest)>}
\def\stripname Time-stamp: <#1: #2 #3 at #4 #5>{#3/#4 (#1)}
\edef\filedate{\expandafter\stripname\timestamp}
\documentclass[a4paper]{amsart}

\DeclareMathSymbol\N \mathord{AMSb}{`N}
\DeclareMathSymbol\Q \mathord{AMSb}{`Q}
\DeclareMathSymbol\R \mathord{AMSb}{`R}
\DeclareMathSymbol\Sor\mathord{AMSb}{`S}
\DeclareMathSymbol\Z\mathord{AMSb}{`Z}
\newcommand\calF{\mathcal{F}}
\newcommand\calN{\mathcal{N}}
\newcommand\calU{\mathcal{U}}
\newcommand\calZ{\mathcal{Z}}
\newcommand\cee{\mathfrak{c}}
\newcommand\cl{\operatorname{cl}}
\newcommand\clbeta[1][X]{\cl_{\beta #1}}
\newcommand\orpr[2]{\langle{#1},{#2}\rangle}
\newcommand\preim{^\gets}
\DeclareMathSymbol\restr\mathbin{AMSa}{"16}
\DeclareMathSymbol\le    \mathrel{AMSa}{"36}
\DeclareMathSymbol\ge    \mathrel{AMSa}{"3E}
\theoremstyle{plain}
\newtheorem{theorem}{Theorem}[section]
\newtheorem{lemma}[theorem]{Lemma}
\newtheorem{corollary}[theorem]{Corollary}
\theoremstyle{definition}
\newtheorem{example}[theorem]{Example}
\theoremstyle{remark}
\newtheorem{question}{Question}
\newtheorem{remark}[theorem]{Remark}
\usepackage{amsrefs}

\begin{document}

\title{$C$-embedding, Lindel\"ofness, and \v{C}ech-completeness}

\author[A. Dow]{Alan Dow}
\address{Department of Mathematics\\
         UNC-Charlotte\\
         9201 University City Blvd. \\
         Charlotte, NC 28223-0001}
\email{adow@charlotte.edu}
\urladdr{https://webpages.uncc.edu/adow}

\author[K. P. Hart]{Klaas Pieter Hart}
\address{Faculty EEMCS\\TU Delft\\
         Postbus 5031\\2600~GA {} Delft\\the Netherlands}
\email{k.p.hart@tudelft.nl}
\urladdr{https://fa.ewi.tudelft.nl/\~{}hart}

\author[J. van Mill]{Jan van Mill}
\address{KdV Institute for Mathematics\\
         University of Amsterdam\\
         P.O. Box 94248\\
         1090~GE {} Amsterdam\\
         The Netherlands}
\email{j.vanmill@uva.nl}
\urladdr{https://staff.fnwi.uva.nl/j.vanmill/}

\author[J. Vermeer]{Hans Vermeer}
\address{Faculty EEMCS\\TU Delft\\
         Postbus 5031\\2600~GA {} Delft\\the Netherlands}
\email{j.vermeer@tudelft.nl}

\dedicatory{In memory of Gary Gruenhage}

\begin{abstract}
We show that in the class of Lindel\"of \v{C}ech-complete spaces
the property of being $C$-embedded is quite well-behaved.
It admits a useful characterization that can be used to show that products
and perfect preimages of $C$-embedded spaces are again $C$-embedded.
We also show that both properties, Lindel\"of and \v{C}ech-complete,
are needed in the product result.
\end{abstract}


\date\filedate

\maketitle

\section*{Introduction}

In \cite{MR4634391}  we investigated whether
in realcompact spaces there could be closed, countable, and discrete
subspaces (closed copies of the space~$\N$ of natural numbers)
that were $C^*$-embedded but not $C$-embedded, or even 
not $C^*$-embedded.
In the fol\-low-up paper \cite{arxiv:2307.07223} we looked for the smallest
power of the real line~$\R$ that could contain such closed copies of~$\N$.

In the present paper we consider more general spaces.
It appears that the members of the class of Lindel\"of \v{C}ech-complete spaces
behave much like~$\N$ as regards $C$-embedding.
Our positive results characterize $C$-embedding and allow us to conclude
that, in this class, $C$-embedding is preserved by products and perfect 
preimages.
We also show, by means of examples, that neither assumption, Lindel\"ofness
nor \v{C}ech-completeness, can be dropped in these results.

\section{Preliminaries}
\label{sec.prelim}

All spaces in this paper are assumed to be, at least, Tychonoff spaces.
The books~\cites{MR0407579,MR1039321} are our primary sources for all 
undefined topological notions.

\subsection{$C$-embedding}
It behoves us to define the central notion of this paper, that of $C$-embedding.

A subspace~$A$ of a space~$X$ is said to be $C$-embedded in~$X$ if every
continuous function from~$A$ to~$\R$ admits a continuous extension to~$X$.

In~\cite{MR0407579}*{Theorem~1.18}, it is shown that $A$~is 
$C$-embedded in~$X$ if and only if
\begin{enumerate}
\item it is $C^*$-embedded in~$X$, that is, every \emph{bounded} continuous
      function from~$A$ to~$\R$ admits a continuous extension to~$X$, 
      \label{item.Cstar} and
\item every zero-set $Z$ that is disjoint from~$A$ is 
      \emph{completely separated}
      from~$A$, that is, there is a continuous function $f:X\to\R$
      such that $f(a)=0$ when $a\in A$ and $f(z)=1$ when $z\in Z$. 
      \label{item.compsep} 
\end{enumerate}

We shall use this equivalence in our proofs as well as the following
characterization of~\eqref{item.Cstar}: if $Z_1$ and $Z_2$ are disjoint 
zero-sets of~$A$ then they are completely separated in~$X$,
see~\cite{MR0407579}*{Theorem~1.17}.

Also, given this equivalence one can weaken~\eqref{item.Cstar} to:
$A$~is $z$-embedded, meaning that for every zero-set~$Z$ of~$A$, there
is a zero-set~$Z^+$ of~$X$ such that $Z=A\cap Z^+$.
The point is that if $Z_1$ and $Z_2$ are disjoint zero-sets of~$A$ then
the intersection $Z_1^+\cap Z_2^+$ is a zero-set that is disjoint from~$A$;
then~\eqref{item.compsep} lets us make~$Z_1^+$ and~$Z_2^+$ a bit smaller
so that they become disjoint. 

\smallbreak
Below we freely use the diagonal embedding~$e$ of a space~$X$ into~$\R^{C(X)}$,
defined by $e(x)=\bigl<f(x):f\in C(X)\bigr>$.
It is well known that $e[X]$~is $C$-embedded in the product and that
$e[X]$~is closed whenever $X$~is realcompact.
Our examples are all realcompact, either because they are Lindel\"of,
or discrete and of small enough cardinality.
We recommend \cite{MR0407579}*{Chapter~8} and 
\cite{MR1039321}*{Section~3.11} for basic information on realcompactness.

\subsection{Rationals and irrationals}

A few of the examples in Section~\ref{sec.examples} use some facts about
the spaces of rational and irrational numbers, and completely metrizable
spaces, that we record here.

We let $\calN$ denote the zero-dimensional Baire space~$\N^\N$,
the product of countably many copies of the discrete space~$\N$,
denoted $B(\aleph_0)$ in~\cite{MR1039321}*{Example~4.2.12}.
In~\cite{MR1039321}*{Exercises~4.3.G and~4.3.H} we find the results that 
we shall use below: $\calN$~is homeomorphic to the subspace of irrational
numbers in~$\R$, and any two countable dense subsets of~$\R$ can be mapped 
to each other by an autohomeomorphism of~$\R$.

\subsection{Two technical results}\label{subsec.special}

In our examples we use Lavrentieff's theorem, 
\cite{MR1039321}*{Theorem~4.3.21}, which states that
if $X$ and $Y$ are completely metrizable, with subspaces~$A$ and~$B$ 
respectively, and $f:A\to B$ is a homeomorphism
then $f$~has an extension to a homeomorphism $\tilde f:\tilde A\to\tilde B$,
where $\tilde A$ and $\tilde B$ are $G_\delta$-sets.

\smallbreak
We also use a result due to various authors,
\cite{MR1039321}*{Problem~2.7.12\,(d)}: 
Let $\kappa$~be an infinite cardinal and let $f:X\to\R$ be continuous,
where $X$~is a product of a sequence $\langle X_\alpha:\alpha<\kappa\rangle$
of separable spaces.
Then there are a countable subset~$E$ of~$\kappa$ and a continuous function
$g:\prod_{\alpha\in E}X_\alpha\to\R$ such that $f=g\circ\pi_E$, 
where $\pi_E:X\to\prod_{\alpha\in E}X_\alpha$ is the projection 
--- in words: $f$~factors through a countable subproduct.

\section{Positive results}

We begin by giving an external characterization of closed subspaces
of Tychonoff spaces that are both Lindel\"of and $C$-embedded.

The following lemma characterizes $C$-embeddedness for arbitrary closed 
subsets.

\begin{lemma}\label{lemma.1st}
Let $A$ be a closed subset of a space $X$.
Then the following three conditions are equivalent.
\begin{enumerate}
\item $A$ is $C$-embedded in~$X$.
\item $A$ is $z$-embedded in~$X$ and for every zero-set $Z$ of~$\clbeta A$ 
      that is disjoint from~$A$ there is a zero-set~$Z^+$ of~$\beta X$ that 
      is disjoint from~$X$ and such that $Z=Z^+\cap\clbeta A$.
\item $A$ is $z$-embedded in~$X$ and for every zero-set $Z$ of~$\clbeta A$ 
      that is disjoint from~$A$ there is a countable family~$\calZ$ of 
      zero-sets of~$\beta X$ such that 
      $Z\subseteq\bigcup\calZ\subseteq\beta X\setminus X$.
\end{enumerate}
\end{lemma}

\begin{proof}
To prove that (1) implies~(2) we take a zero-set $Z$ of~$\clbeta A$ and
construct~$Z^+$, as follows.
Let $f:\clbeta A\to[0,1]$ be continuous such that $Z=\{a\in\clbeta A:f(a)=0\}$
and consider its restriction $f\restr A$ to~$A$; 
this is a function from~$A$ to~$(0,1]$.
By $C$-embeddedness we have an extension~$F:X\to(0,1]$ of~$f\restr A$, which
we then extend to $\beta F:\beta X\to[0,1]$.
Then let $Z^+$ be the zero-set of~$\beta F$.

\smallskip
That (2) implies (3) is clear so we turn to proving that (3) implies~(1).

We already know that $A$~is $z$-embedded in~$X$, so let $Z$~be a zero-set 
of~$X$ that is disjoint from~$A$; we show that $Z$ and~$A$ are completely 
separated.
Let $Z_\beta$ be a zero-set of~$\beta X$ such that $Z=X\cap Z_\beta$ and
let $Z_A=Z_\beta\cap\clbeta A$.
Then $Z_A$~is a zero-set of~$\clbeta A$ that is disjoint from~$A$.

Let $\calZ$ be a countable family of zero-sets of~$\beta X$ as in the 
assumption.
Say $\calZ=\{Z_n:n\in\omega\}$, and for every~$n$ 
let $C_n=\beta X\setminus Z_n$ be the complementary cozero-set.

Then $X\subseteq L=\bigcap_{n\in\omega}C_n$, and by~\cite{MR0133350}*{Lemma~2.2} 
(or \cite{MR1039321}*{Exercise~3.8.F}) 
the space $L$~is Lindel\"of.
In addition the sets $Z_\beta\cap L$ and $\clbeta A\cap L$ are closed
and disjoint in~$L$; as $L$~is normal these sets are completely separated 
in~$L$, and so $Z$ and~$A$ are completely separated in~$X$, 
because $X\subseteq L\subseteq\beta X$. 
\end{proof}

Using this lemma we get our principal result.

\begin{theorem}\label{thm.first}
Let $A$ be a closed subset of a space $X$.
Then the following three conditions are equivalent.
\begin{enumerate}
\item $A$ is Lindel\"of and $C$-embedded in~$X$.
\item For every compact subset $K$ of $\clbeta A\setminus A$ there is a 
      zero-set~$Z$ of~$\beta X$ such that 
      $K\subseteq Z\subseteq\beta X\setminus X$.
\item For every compact subset $K$ of $\clbeta A\setminus A$ there is a 
      countable family $\calZ$ of zero-sets of~$\beta X$ such that 
      $K\subseteq\bigcup\calZ\subseteq\beta X\setminus X$.
\end{enumerate}
\end{theorem}

\begin{proof}
To prove that (1) implies (2) we let $K$ be a compact subset 
of $\clbeta A\setminus A$ and find a zero-set $Z$ of $\beta X$ such 
that $K\subseteq Z\subseteq\beta X\setminus X$.

To begin we choose for every $a\in A$ a continuous function 
$f_a:\beta X\to[0,1]$ such that $f_a(a)=1$ and $f_a(x)=0$ if $x\in K$.
For each $a$ we let $U_a=f\preim\bigl[(\frac12,1]\bigr]$.
There is a countable subset $\{a_n:n\in\omega\}$ of~$A$ such 
that $A\subseteq\bigcup_{n\in\omega}U_{a_n}$.

Let $g=\sum_{n\in\omega}2^{-n}f_{a_n}$.
Then $g$~is continuous, $g(a)>0$ when $a\in A$, and $g(x)=0$ when $x\in K$.
We would like to let $Z=g\preim(0)$, but that set may intersect~$X$.
However, $S=\{x\in X:g(x)=0\}$ is a zero-set of~$X$ that is disjoint from~$A$.
As $A$~is $C$-embedded in~$X$ the sets $S$ and $A$ are completely separated.
Let $f:X\to[0,1]$ be continuous such that $f(x)=1$ if $x\in S$ and $f(a)=0$
if $a\in A$.
Note that $\beta f$ vanishes on $\clbeta A$ and in particular on~$K$.

Now let $h=g+\beta f$.
Then $h(x)\ge g(x)>0$ when $x\in X\setminus S$ and
$h(x)\ge1$ when $x\in S$.
Also, $h(x)=0$ when $x\in K$.
It follows that $h\preim(0)$ is the zero-set of~$\beta X$ that we seek.

\medskip
Clearly (2) implies (3).

\medskip
We finish by proving that (3) implies (1).   
To begin: the present condition~(3) is stronger than the second part of~(3)
in Lemma~\ref{lemma.1st}.
We need to show that $A$~is Lindel\"of and $z$-embedded in~$X$.
It will actually be simpler to show that $A$~is $C^*$-embedded in~$\clbeta A$.

\smallskip
That $A$~is Lindel\"of is proved as follows.
Let $\calU$ be a collection of open subsets of~$\beta X$ that covers~$A$.
Let $K=\clbeta A\setminus\bigcup\calU$ and let $\calZ$ be a countable
family of zero-sets of~$\beta X$ as in the assumption.

As in the proof of Lemma~\ref{lemma.1st} the complement~$L$ of
$\bigcup\calZ$ in~$\beta X$ is Lindel\"of and it contains~$X$.
Then $L\cap\clbeta A$ is Lindel\"of as well and, moreover,
contained in~$\bigcup\calU$.
But $A\subseteq L\cap\clbeta A$, so $A$~is covered by a countable
subfamily of~$\calU$.

\smallskip
To see that $A$~is $C^*$-embedded in~$\clbeta A$, and hence in~$\beta X$,
we show that if $E$ and $F$ are disjoint closed subsets of~$A$ then their 
closures in~$\clbeta A$ are disjoint; this shows that $\clbeta A$ actually
\emph{is} the \v{C}ech-Stone compactification of the normal space~$A$.

\smallskip
Let $E$ and~$F$ be as above and let $K=\clbeta E\cap\clbeta F$.
Then $K\subseteq\clbeta A\setminus A$ and hence there is a
countable family~$\calZ$ of zero-sets of~$\beta X$ as in our assumption.

We have just established that $L=\beta X\setminus\bigcup\calZ$ 
is Lindel\"of, hence $L$~is normal as well.
Also $X\subseteq L\subseteq\beta X$, and so $\beta L=\beta X$.

In addition we have $\cl_LE\cap\cl_LF=\emptyset$ and hence 
$\clbeta E\cap\clbeta F=\emptyset$ (so in hindsight $K=\emptyset$).
\end{proof}

There are two special cases of this result that are worth recording here.
They consider Lindel\"of subspaces that are locally compact or 
\v{C}ech-complete.

\begin{theorem}
Let $A$ be a closed and locally compact subset of~$X$.
Then the following three conditions are equivalent.
\begin{enumerate}
\item $A$ is Lindel\"of and $C$-embedded in~$X$.
\item There is a zero-set~$Z$ of~$\beta X$ such that 
      $\clbeta A\setminus A\subseteq Z\subseteq\beta X\setminus X$.
\item There is a countable family $\calZ$ of zero-sets of~$\beta X$ such that 
      $\clbeta A\setminus A\subseteq\bigcup\calZ\subseteq\beta X\setminus X$.
\end{enumerate}
\end{theorem}

\begin{proof}
The set $\clbeta A\setminus A$ is closed and hence compact.
Therefore it is necessary and sufficient to assume or establish~(2) 
and~(3) for that set only.  
\end{proof}

\begin{theorem}\label{thm.Cech}
Let $A$ be a closed and \v{C}ech-complete subset of~$X$.
Then the following conditions are equivalent.
\begin{enumerate}
\item $A$ is Lindel\"of and $C$-embedded in~$X$.
\item There is a countable family $\calZ$ of zero-sets of~$\beta X$ such that 
      $\clbeta A\setminus A\subseteq\bigcup\calZ\subseteq\beta X\setminus X$.
\end{enumerate}
\end{theorem}

\begin{proof}
By the definition of \v{C}ech-completeness
the set $\clbeta A\setminus A$ is an $F_\sigma$-subset of~$\clbeta A$.
One applies (2) or~(3) in Theorem~\ref{thm.first} to the countably many
closed sets whose union is~$\clbeta A\setminus A$ to obtain the desired
cover.
\end{proof}

From these characterizations we deduce two results about the preservation
of $C$-embeddedness.

\begin{theorem}\label{thm.product}
Let $\langle X_i:i<k\rangle$ be a sequence of spaces, 
where $k$~is a finite ordinal or~$\omega$, 
and let $\langle A_i:i<k\rangle$ be a corresponding sequence of 
$C$-embedded subspaces ($A_i$ of~$X_i$) that are closed, 
Lindel\"of, and \v{C}ech-complete.
Then the product $\prod_{i<k}A_i$ is 
Lindel\"of, \v{C}ech-complete, and $C$-embedded in~$\prod_{i<k}X_i$.
\end{theorem}

\begin{proof}
We write $A=\prod_{i<k}A_i$ and $X=\prod_{i<k}X_i$.

For each~$i$ let $\calZ_i$ be a countable family of zero-sets in~$\beta X_i$ 
such that
$$
\clbeta[X_i] A_i\setminus A_i
       \subseteq\bigcup\calZ_i\subseteq\beta X_i\setminus X_i.
$$
Then $\cl A\setminus A$ is covered by union $\calZ$ of the families
$\{\pi_i\preim[Z]:Z\in\calZ_i\}$.
These are countable families of zero-sets in~$\prod_{i<k}\beta X_i$
and their members are contained 
in $(\prod_{i<k}\beta X_i)\setminus X$.

Let $f:\beta X\to\prod_{i<k}\beta X_i$ be the natural map.
Then $\{f\preim[Z]:Z\in\calZ\}$ is a countable family of zero-sets 
in $\beta X$.
Because $f$~is perfect its union~$\bigcup\calZ$ is contained in
$\beta X\setminus X$ and it 
contains $\clbeta A\setminus A$.  

Theorem~\ref{thm.Cech} implies that $A$ is 
Lindel\"of and $C$-embedded in~$X$.
\end{proof}

By strengthening the assumptions on the subspaces and weakening the conclusion
we get a version of this result for arbitrary products.

\begin{corollary}
Let $\langle X_\alpha:\alpha<\kappa\rangle$ be an arbitray sequence of spaces, 
and let $\langle A_\alpha:\alpha<\kappa\rangle$ be a corresponding sequence of 
$C$-embedded subspaces ($A_\alpha$ of~$X_\alpha$) that are closed, 
Lindel\"of, \v{C}ech-complete, and separable.
Then the product $A=\prod_{\alpha<\kappa}A_\alpha$ is 
$C$-embedded in~$X=\prod_{\alpha<\kappa}X_\alpha$.
\end{corollary}

\begin{proof}
If $f:A\to\R$ is continuous then, by separability of the factors,
the factorization result from section~\ref{subsec.special} 
implies that $f$~factors through a countable 
subproduct~$\prod_{\alpha\in E}A_\alpha$.
The previous theorem implies that the factored map~$f_E$ has a continuous
extension~$F$ to~$\prod_{\alpha\in E}X_\alpha$.
Then $F$~determines a continous extension of~$f$ to~$X$.
\end{proof}

\begin{theorem}\label{thm.preim}
Let $A$ be a closed, Lindel\"of and \v{C}ech-complete subspace of $X$ 
that is also $C$-embedded and let $f:Y\to X$ be a perfect surjection.
Then $f\preim[A]$ is $C$-embedded in~$Y$.
\end{theorem}
\begin{proof}
The previous proof applies.
If $\calZ$~is a countable family of zero-sets of~$\beta X$ such that
$$
\clbeta A\setminus A\subseteq\bigcup\calZ\subseteq\beta X\setminus X
$$
then, because $f$~is perfect, we have
$$
\clbeta[Y] f\preim[A]\setminus f\preim[A]\subseteq
\bigcup\{f\preim[Z]:Z\in\calZ\}\subseteq\beta Y\setminus Y.\qedhere
$$ 
\end{proof}
\begin{remark}
The proofs that (1) implies~(2) and (3) implies~(1) in Theorem~\ref{thm.first}
use properties of~$\beta X$.

The proof of Theorem~\ref{thm.product} shows, implicitly, that if $A$ satisfies
condition~(3) in \emph{some} compactification of~$X$ then it will satisfy
that condition in~$\beta X$ as well.
The converse of this does not hold.

For let $X$ be an uncountable discrete space, and $A$ a countable subset 
of~$X$.
Then $A$~is --- trivially --- Lindel\"of, \v{C}ech-complete, and $C$-embedded 
in~$X$, and, in addition, $\clbeta A\setminus A$ is itself a zero-set 
of~$\beta X$.

But $A$ does not satisfy condition~(3) in the one-point 
compactification~$\alpha X$ of~$X$.
Indeed $\cl_{\alpha X}A\setminus A=\alpha X\setminus X=\{\infty\}$, where 
$\infty$~is the point at infinity, and this remainder contains no zero-set
of~$\alpha X$.
\end{remark}

\section{Examples}
\label{sec.examples}

An easy consequence of Theorem~\ref{thm.product} is that if two spaces~$X$
and~$Y$ contain closed copies, $N_1$ and $N_2$ respectively, of~$\N$ that
are $C$-embedded then the product $N_1\times N_2$ is $C$-embedded 
in~$X\times Y$. 

This can also be established in an elementary way.
There are continuous functions $f_1:X\to\R$ and $f_2:Y\to\R$ such that
$f_1$ maps~$N_1$ injectively into $\{2^n:n\in\N\}$ and
$f_2$ maps~$N_2$ injectively into $\{3^n:n\in\N\}$.
Then $f:X\times Y\to\R$, defined by $f(x,y)=f_1(x)\cdot f_2(y)$,
maps $N_1\times N_2$ injectively into~$\N$ and this suffices to ensure 
$C$-embedding.  

The countability of~$\N$ corresponds to the Lindel\"of assumption
in Theorem~\ref{thm.product}.
This assumption cannot be dropped completely as the next example shows.

\begin{example}\label{example.discrete}
Let $\kappa$ be a cardinal such that there is an uncountable closed and 
discrete subset~$D$ that is $C$-embedded in~$\R^\kappa$.
Then $D\times D$ is not $C^*$-embedded in~$\R^\kappa\times\R^\kappa$.
\end{example}

\begin{remark}
We can have $D$ of cardinality~$\aleph_1$ and with $\kappa=2^{\aleph_1}$.
For if $\lambda$~is less than the first measurable cardinal then $\lambda$,
with its discrete topology, is realcompact and so the image of~$\lambda$
under the diagonal map $e:\lambda\to\R^{C(\lambda)}$ is closed and $C$-embedded.
\end{remark}

\begin{proof}[Proof of Example~\ref{example.discrete}]
We define $f:D\times D\to[0,1]$ as follows:
\begin{itemize}
\item $f(d,e)=0$ if $d\neq e$, and
\item $d\mapsto f(d,d)$ maps into the interval $(0,1]$, 
     injectively if $|D|\le\cee$, and  
     surjectively if $|D|\ge\cee$ (and so bijectively if $|D|=\cee$).
\end{itemize}
Because $D\times D$ is discrete $f$~is automatically continuous.

Now assume $F:\R^\kappa\times\R^\kappa\to\R$ is a continuous extension of~$f$ and
let $C$ be a countable subset of~$\kappa$ such that $F$~factors through 
$\R^C\times\R^C$.
So we have a continuous map~$g:\R^C\times\R^C\to[0,1]$ such 
that $F=g\circ(\pi\times\pi)$ where $\pi:\R^\kappa\to\R^C$ ~is the projection.

Let $E=\pi[D]$ and observe that $E$~is uncountable.
If $|D|\le\cee$ then $\pi$~is a bijection between $D$ and~$E$ because
$d\mapsto g(\pi(d),\pi(d))$ is injective, and if $|D|\ge\cee$ 
then $g[E\times E]=[0,1]$, so $|E|\ge\cee$ as well.

Now let $e\in E$, then $g(e,e)>0$ and so there is a neighbourhood $U$ of~$e$
in~$\R^C$ such that $g(x,y)>0$ for all $x,y\in U$.
We claim that $U\cap E=\{e\}$.
Indeed, if $x\in U$ and $x\neq e$ then $g(x,e)>0$ whereas $g(y,e)=0$ whenever
$y\in E$ and $y\neq e$.

It follows that $E$~is an uncountable relatively discrete subset of the
separable and metrizable space~$\R^C\times\R^C$, a clear impossibility.
\end{proof}

The assumption that the factors be \v{C}ech-complete cannot be dropped either.

\begin{theorem}
The product $\Q\times\Q$ is not $C^*$-embedded in $\R^{C(\Q)}\times\R^{C(\Q)}$,
where $\Q$~is embedded in $\R^{C(\Q)}$ via the diagonal embedding
$e:\Q\to\R^{C(\Q)}$. \label{thm.QxQ}
\end{theorem}

Before we give the proof we need a lemma first.

\begin{lemma}\label{X.extend}
Let $X$ be a separable and metrizable space, and 
let $f:X\times X\to\R$ be a continuous function.
Then the following two statements about~$f$ are equivalent.
\begin{enumerate}
\item $f$~has a continuous extension to the 
      product\/ $\R^{C(X)}\times\R^{C(X)}$, where we identify $X$ with its 
      image~$e[X]$ under the diagonal embedding $e:X\to\R^{C(X)}$,
      and
\item there is a completely metrizable extension~$M$ of~$X$ such that 
      $f$~has a continuous extension to~$M\times M$.
\end{enumerate}
\end{lemma}

\begin{proof}
Necessity: assume $F:\R^{C(X)}\times\R^{C(X)}\to\R$ is an extension of~$f$.
We can find a countable subset~$E$ of~$C(X)$ such that $F$~factors through
the partial product $\R^E\times\R^E$. 
We can enlarge~$E$, if need be, so that the projection 
$\pi_E:\R^{C(X)}\to\R^E$ is a homeomorphism on~$e[X]$, that is,
$\pi_E\circ e:X\to\R^E$ is a homeomorphic embedding;
here is where we use that $X$~is regular and second-countable.

Now let $G:\R^E\times\R^E\to\R$ be such that $F=G\circ(\pi_E\times\pi_E)$.
Then $\R^E$ is a completely metrizable extension of~$X$, and $G$~is a 
continuous extension of~$f$.  

\smallskip
Sufficiency: assume $M$ is a completely metrizable space that contains~$X$
and such that there is a continuous extension $g:M\times M\to\R$ of~$f$.
We assume $X$~is dense in~$M$ as its closure in~$M$ is completely metrizable
as well.

Then there is an embedding $h:M\to\R^\omega$ such that $h[M]$~is closed.
Because $h[M]\times h[M]$ is closed in $\R^\omega\times\R^\omega$ there is 
a continuous function~$G:\R^\omega\times\R^\omega\to\R$ that extends~$g$
(more precisely: such that $g=G\circ h$).
The countably many projections $\pi_n:\R^\omega\to\R$ yield members of~$C(M)$
via $p_n=\pi_n\circ h$.
These give us a projection $\Pi:\R^{C(X)}\to\R^\omega$ such that 
$\Pi\circ e=h$ on~$X$.
Then $G\circ(\Pi\times\Pi)$ is the extension of~$f$ 
to $\R^{C(X)}\times\R^{C(X)}$.
\end{proof}

\begin{proof}[Proof of Theorem~\ref{thm.QxQ}]
By the lemma, to show that $\Q\times\Q$ is not $C^*$-embedded in
$\R^{C(\Q)}\times\R^{C(\Q)}$ it suffices to exhibit a bounded continuous function
$f:\Q\times\Q\to\R$ that has no continuous extension to $M\times M$
whenever $M$~is a completely metrizable extension of~$\Q$.

We claim that it suffices to find a continuous function $f:\Q\times\Q\to\R$
such that there is no $G_\delta$-subset~$G$ of~$\R$ such that $f$~has
a continuous extension to~$G\times G$.
Indeed, if $M$~is an arbitrary completely metrizable extension of~$\Q$,
say with embedding $g:\Q\to M$, then Lavrentieff's theorem yields $G_\delta$-sets
$A$ in~$\R$ and $B$ in~$X$, and a homeomorphism $\bar g:A\to B$ that 
extends~$g$. 
If $\bar f$~were a continuous extension of~$f$ to~$M\times M$ then
$\bar f\circ (\bar g\times\bar g)$ would be a continuous extension of~$f$
to~$A\times A$.

To define~$f$ we let $L$ be the line in the plane with equation $y=x+\pi$.
Clearly $L$~is disjoint from~$\Q\times\Q$.
But, if $A$~is a $G_\delta$-subset of~$R$ that contains~$\Q$ 
then $(A\times A)\cap L\neq\emptyset$.
For let $A$ be such a $G_\delta$-set then both $A$ and $A-\pi$ are dense
$G_\delta$-subsets of~$\R$ and hence, by the Baire Category theorem
the intersection $B=A\cap(A-\pi)$ is also a dense $G_\delta$-set.
But if $x\in B$ then $(x,x+\pi)\in (A\times A)\cap L$.

Now define $f:\Q\times\Q\to[-1,1]$ by
$$
f(p,q)=\frac{q-p-\pi}{|q-p-\pi|}  
$$
Then $f$ has no continuous extension to any point of~$L$.
\end{proof}

\begin{example}
One may wonder whether Theorem~\ref{thm.QxQ} can be proved using a 
homeomorphism between $\Q\times\Q$ and $\Q$.
The idea is that such a homeomorphism should change the geometry of
$\Q\times \Q$ too much to allow it to be extended to $\R^{C(\Q)}\times\R^{C(\Q)}$.

\medskip
We exhibit two homeomorphisms between $\Q\times\Q$ and $\Q$: one that can be 
extended and one that cannot.

\smallskip
The first comes via a direct application of Lemma~\ref{X.extend}.

We let $\calN=\N^\N$ be zero-dimensional Baire space and let $\Q$ be embedded 
in~$\calN$ as the subset~$Q$ of sequences that end in zeros,
so $Q=\{x:(\exists m)(\forall n\ge m)(x_n=0)\}$.

Now, the homeomorphism $h:\calN\times\calN\to\calN$ is obtained by interleaving
sequences and it maps $Q\times Q$ to~$Q$.
If we compose this map with an embedding~$g$ of~$\calN$ into~$\R$ that maps~$Q$
onto~$\Q$ then $g\circ h$~is an extension to~$\calN\times\calN$
of its restriction to~$Q\times Q$.

\smallskip
To obtain the second homeomorphism we take the embedding $g:\calN\to\R$
used above with $e[Q]=\Q$ and we let $N=g[\calN]$.

Then $N$~is a $G_\delta$-set that contains~$\Q$ and the composition
$G=g\circ h\circ(g^{-1}\times g^{-1})$ is a homeomorphism from $N\times N$ 
to~$N$.
In particular, $G$~is injective on the intersection of~$N\times N$
with the line~$L$ above.

Next define a homeomorphism $f:\Q\times\Q\to\Q\times\Q$ 
by $f(p,q)=(p+1,q+1)$ if $q>p+\pi$ and $f(p,q)=(p-1,q-1)$ if $q<p+\pi$.
Then the composition $G\circ f$ is still a homeomorphism between $\Q\times\Q$
and~$\Q$.
However, if $M$~is a $G_\delta$-set that contains~$\Q$ then 
$L\cap((M\cap N)\times(M\cap N))$ is nonempty and
$G\circ f$~has no continuous extension to any point in that intersection.

For let $\orpr xy$ be a point in the intersection and take two sequences 
$\bigl<\orpr{p_n}{q_n}:n\in\omega\bigr>$ and
$\bigl<\orpr{r_n}{s_n}:n\in\omega\bigr>$ in $\Q\times\Q$ that converge 
to~$\orpr xy$ and such that $q_n>p_n+\pi$ and $s_n<r_n+\pi$ for all~$n$.
Then $\lim_n(G\circ f)(p_n,q_n)=G(x+1,y+1)$
and $\lim_n(G\circ f)(r_n,s_n)=G(x-1,y-1)$.
It follows that $G\circ f$ cannot be extended to $M\times M$. \qed
\end{example}

The space~$\Q$ is very not \v{C}ech-complete.
It is natural to wonder how close to \v{C}ech-complete a separable and 
metrizable can be and still satisfy Theorem~\ref{thm.QxQ}.

We can re-use the proof of Theorem~\ref{thm.QxQ} (and Lemma~\ref{X.extend})
to get an example that is a Baire space.

\begin{example}
We take a subspace~$A$ of~$\R$ such that $\{x+\pi:x\in A\}=\R\setminus A$.
That such a space exists was established by Van Mill in~\cite{MR0660152}
in an alternative proof of Menu's theorem from~\cite{Menu} that $\R$ can 
be partitioned into two mutually homeomorphic and homogeneous subspaces. 

\smallbreak
A particularly transparent construction of a set $A$ as required was 
suggested by Jeroen Bruyning in~\cite{MR0660152}. 
Let $H$ be a Hamel base for~$\R$ over the field~$\Q$ that 
contains~$1$ and~$\pi$.

For $x\in\R$ let $x_\pi$ denote its (rational) $\pi$-coordinate with respect
to this base. 
Let $A=\{x\in\R:\lfloor x_\pi\rfloor$~is even$\}$, where 
$\lfloor\cdot\rfloor$~denotes the greatest-integer function.

\smallbreak
By the Baire category theorem the set $A$ is a Baire space.
By construction its square $A\times A$ is disjoint from the line~$L$
with equation $y=x+\pi$.
As before the function $f:A\times A\to[-1,1]$ defined by
$$
f(a,b)=\frac{b-a-\pi}{|b-a-\pi|}  
$$
has no continuous extension to any point of~$L$.

Finally let $G$ be a $G_\delta$-set that contains~$A$, 
say $G=\bigcap_{n=1}^\infty O_n$ with each $O_n$ open in~$\R$.
Since $A$~is dense (it contains~$\Q$) we find that for every~$n$ the
set $O_n$ is dense in~$\R$ and hence the difference $O_n\setminus A$ is 
dense in~$\R\setminus A$.
As $\R\setminus A$ is a Baire space we deduce that $G\setminus A$~is nonempty.
Also, $G-\pi$ contains $\R\setminus A$ so that 
$H=(G\setminus A)\cap(G-\pi)\neq\emptyset$.
As above, if $x\in H$ then $(x,x+\pi)$ is in $(G\times G)\cap L$.
\end{example}

Thus the product theorem does not (even) hold for Baire spaces that are 
separable and metrizable.

\smallbreak
There are various properties that are shared by locally compact spaces
and completely metrizable spaces, and that imply that the space is a Baire
space, see~\cite{MR0380745}.
One can ask for each of these properties whether satisfy our preservation 
theorems.

Many of these properties have in common that for (separable) metrizable spaces
they imply \v{C}ech-completeness; this means that counterexamples will
have to be non-metrizable.

\smallbreak
We show that the property of co-compactness, see also~\cite{MR0261539},
is not strong enough to guarantee that Theorem~\ref{thm.product} holds.

A space $(X,\tau)$ is \emph{co-compact} if there is a family~$\calF$ 
of $\tau$-closed sets whose $\tau$-interiors form a base for the 
given topology~$\tau$, and that at the same time forms a subbase for the
closed sets of a compact topology~$\tau_\calF$ on the set~$X$.
We note that $\tau_\calF$~will be $T_1$ but will not necessarily be Hausdorff.

In~\cite{MR0261539} it is shown that for metrizable spaces co-compactness
and \v{C}ech-com\-plete\-ness are equivalent.

\smallbreak
It turns out that the subject of Gary's first paper~\cite{MR0317286}, the
Sorgenfery line~$\Sor$, satisfies Theorem~\ref{thm.QxQ} too.
It is well-known that $\Sor$~is Lindel\"of, and it is readily seen to
be co-compact: let $\calF$ be the family of closed and bounded 
intervals in~$\R$.

The proof of Theorem~\ref{thm.QxQ} for~$\Sor$ rests on the following 
observation about the subset $D=\{\orpr xy:x+y\ge0\}$ of the plane.

\begin{lemma}
Let $\tau$ be a topology on~$\R$ such that $D$~is open in the plane with
respect to the product topology from~$\tau$.
Then for every $a\in\R$ the set $[a,\infty)$ belongs to~$\tau$, and
hence $\tau$~is not second-countable.  
\end{lemma}

\begin{proof}
Let $x\in\R$ and let $U$ and $V$ be members of~$\tau$ such that 
$\orpr x{-x}\in U\times V\subseteq D$.
We claim that $U\subseteq[x,\infty)$ (and by symmetry $V\subseteq[-x,\infty)$).
Indeed, let $z\in U$, then $\orpr z{-x}\in U\times V\subseteq D$
and hence $z-x\ge0$, or $z\ge x$.

Let $\langle U_x:x\in\R\rangle$ be a choice function from~$\tau$ such that
$x\in U_x\subseteq[x,\infty)$ for all~$x$.
Then $[a,\infty)=\bigcup_{x\ge a}U_x$ is in~$\tau$, for every~$a$.  

That $\tau$~is not second-countable follows as in the familiar proof that
the Sorgenfrey line is not second-countable.
\end{proof}

\begin{proof}[Proof of Theorem~\ref{thm.QxQ} for~$\Sor$]
Let $f:\Sor\times\Sor\to\R$ be the characteristic function of~$D$
and assume $F:\R^{C(\Sor)}\times\R^{C(\Sor)}\to\R$ is a continuous extension
of~$F$.
We take a countable subfamily~$E$ of~$C(\Sor)$ and a continuous function
$G:\R^E\times\R^E\to\R$ such that $F=G\circ(\pi_E\times\pi_E)$.
If we make sure that the identity function $i:\Sor\to\R$ belongs to~$E$ then
$\pi_E$ is injective on~$\Sor$ and 
$G$~restricts to the characteristic function of~$D$ on~$\Sor\times\Sor$.

It follows that $D$~is open in the topology~$\tau$ that $\Sor\times\Sor$ 
inherits from~$\R^E\times\R^E$, but the subspace~$(\Sor,\tau)$  
is second-countable.
\end{proof}

\end{document}